\documentclass[12pt, reqno]{amsart}
\usepackage{amsmath, amsthm, amscd, amsfonts, amssymb, graphicx, color}
\usepackage[bookmarksnumbered, colorlinks, plainpages]{hyperref}

\newtheorem{theorem}{Theorem}[section]
\newtheorem{lemma}[theorem]{Lemma}
\newtheorem{proposition}[theorem]{Proposition}

\theoremstyle{definition}
\newtheorem{definition}[theorem]{Definition}

\theoremstyle{remark}
\newtheorem{remark}[theorem]{Remark}
\numberwithin{equation}{section}

 \def\R{{\bf{R}}}
\def\C{{\bf{C}}}

\def\N{{\bf{N}}}

\def\defequal{\stackrel{\mathrm{def}}{=}}

\def\ra{\rightarrow}

\def\sk{\smallskip}

\def\tfae{the following conditions are equivalent:}

\begin{document}
\setcounter{page}{1}


\title[Cocycles and derivations]{Groupoid cocycles and derivations}

\author[Jean Renault]{Jean Renault}

\address{$^{1}$ D\'epartment de Math\'ematiques, Universit\'e d'Orl\'eans,
45067 Orl\'eans, France.}
\email{\textcolor[rgb]{0.00,0.00,0.84}{Jean.Renault@univ-orleans.fr}}





\begin{abstract}
This is a study of derivations constructed from conditionally negative type functions on groupoids which illustrates Sauvageot's theory of non-commutative Dirichlet forms.
\end{abstract} \maketitle

\section{Introduction and preliminaries}

\noindent The relation between group cocycles and derivations is well-known. B. Johnson's memoir \cite{joh:cohomology} is a fundamental reference for this in the framework of Banach algebras. While the classical theory  deals mostly with bounded derivations, the occurence of unbounded derivations came quickly (see for example \cite{ln:der}). Groupoid cocycles were also considered at an early stage, mainly in a measure-theoretical setting. It is observed in \cite{ren:approach} that a continuous real cocycle on a topological groupoid $G$, i.e. a continuous groupoid homomorphism $c:G\ra \R$, defines by pointwise multiplication a derivation on the C*-algebra of the groupoid (when one assumes that $G$ is locally compact and endowed with a Haar system). It is shown that this derivation is closable and that it is bounded if and only if the cocycle is bounded. In this article, we shall replace the vector space $\R$ (a trivial $G$-module) by a $G$-real Hilbert bundle. A motivation comes from Delorme and Guichardet's characterization of Kazhdan's property T for a group $G$ as a fixed point property for affine isometric actions of $G$. This is equivalent to the vanishing of the 1-cohomology group of $G$ with coefficients in an orthogonal representation of $G$. As noticed by J.-L. Tu in \cite{tu:conjecture}, the dictionary between affine isometric actions, 1-cocycles and functions conditionally of negative type still holds for a groupoid $G$, when one considers a $G$-real Hilbert bundle instead of a single $G$-real Hilbert space. We recall this dictionary in the next section, using slightly different topological assumptions and giving more complete results than those of \cite{tu:conjecture}. The next step is a C*-algebraic construction. A $G$-Hilbert bundle naturally defines a C*-correspondence of the C*-algebra of the groupoid to itself. This is a well-known construction for groups (where the authors usually construct a von Neumann correspondence rather than a C*-correspondence). This construction also appears in P.-Y. Le Gall's thesis \cite{leg:these} as part of the descent functor from $KK_G$ to the Kasparov group of the crossed product. We recall this construction in Section 3 and show that it defines a functor from the category of $G$-Hilbert bundles to the category of $C^*(G)$-correspondences which respects the product (the precise statement is Theorem \ref{functor}). We give in Section 4 the construction of the derivation defined by a cocycle with values in a $G$-Hilbert bundle and show in Theorem \ref{Sauvageot pair} that it agrees with Sauvageot's construction of \cite{sau:tangent}. We also prove that this derivation is closable.

We assume that the reader is familiar with the basic theory of groupoids and their C*-algebras, as given in \cite{ren:approach}. We assume that $G$ is a second countable and locally compact Hausdorff groupoid. We denote by $r,s: G\ra G^{(0})$ the range and source maps. Given $x\in G^{(0)}$, we define $G^x=r^{-1}(x)$. We use the notion of Banach bundle (we shall only consider Hilbert bundles) as defined in \cite{fd:representations}. We also need the notion of an affine Banach bundle; its definition is easily adapted from the definition of a Banach bundle.


\section{Positive type and conditionally negative type functions}
One can expect that the theory of positive type and conditionally negative type functions extends from groups to groupoids. In fact, the most primitive theory does not concern groups  but trivial groupoids $X\times X$, where these functions are usually called kernels. The basic ingredient is the GNS construction. These notions have appeared previously, for example in \cite{{rw:Fourier},{ren:Fourier},{pat:Fourier}} for positive type functions on groupoids and in \cite{tu:conjecture} for positive type and conditionally negative type functions on groupoids. However, the treatment of continuity is incomplete in the first set of references; J.-L. Tu gives in \cite{tu:conjecture} the main results on the subject but only a sketch of the proofs; moreover his assumptions are different from ours (he assumes the existence of continuous local sections for the range map while we assume the existence of a continuous Haar system; although most examples satisfy both assumptions, the existence of a continuous Haar system is the natural assumption in the construction of convolution algebras).

\subsection{Positive type kernels}
Four our purposes, we need to revisit the classical theory of positive type kernels. Given a locally compact Hausdorff space $X$, we denote by $C(X)$ the space of continuous complex valued functions on $X$ and by $M(X)$ the space of complex Radon measures on $X$. We add the subscript $_c$ (i.e. $C_c(X), M_c(X)$) to add the condition of compact support. We define the weak topology on $M_c(X)$ as the locally convex topology defined by the semi-norms $\alpha\mapsto |<f,\alpha>|=|\int f d\alpha|$, where $f\in C(X)$. The complex conjugate of a function $f$ [resp. a measure $\alpha$] is denoted by $f^*$ [resp. $\alpha^*$]. Recall that a complex valued function $\varphi$ on $X\times X$ (also called a kernel) is said to be of positive type if for all $n\in\N^*$ and all $x_1,\ldots,x_n\in X$, the $n$ by $n$ matrix $[\varphi(x_i,x_j)]$ is of positive type. The following result is classical (e.g. \cite[Theorem C.1.4]{bhv:T}  as a recent reference). We give its proof to introduce notations which we shall use subsequently.
\begin{theorem}\label{kernel}
Let $X$ be a locally compact Hausdorff space and let
$\varphi$ be a continuous kernel on $X\times X$
of positive type. Then,
\begin{enumerate}
\item there exists a Hilbert space $E$ and a continuous map $e:X\ra E$ such that for all $x,y\in X$, $\varphi(x,y)=(e(x)| e(y))$ and the linear span of $e(X)$ is dense in $E$.
\item the pair $(E,e)$ is unique, in the sense that if $(\underline E,\underline e)$ is another solution, there exists a unique isomorphism 
$u:E\ra\underline E$ such that $\underline e=u\circ e$.
\end{enumerate}
\end{theorem}

\begin{proof} The Hilbert space can be constructed as follows. The kernel $\varphi$ extends to a sesquilinear form $\Phi$ on $M_c(X)\times M_c(X)$ such that $\Phi(\alpha,\beta)=\langle\varphi,\alpha^*\otimes\beta\rangle$. By assumption, $\Phi(\alpha,\alpha)\ge 0$ when $\alpha$ has a finite support. By density of the measures with finite support in the weak topology (e.g. \cite[ch. III, \S 2, th. 1]{bbki:integration}), this holds for all $\alpha\in M_c(X)$. Thus $\Phi$ is a pre-inner product. After separation and completion, one obtains a Hilbert $E$ and a linear map $j:M_c(X)\ra E$ such that $(j(\alpha)|j(\beta)=\langle\varphi,\alpha^*\otimes\beta\rangle$. We define $e:X\ra E$ by $e(x)=j(\epsilon_x)$, where $\epsilon_x$ is the point mass at $x$. The equality
$$\|e(x)-e(y)\|^2=\varphi(x,x)-2{\rm Re}(\varphi(x,y))+\varphi(y,y)$$
shows that $e$ is continuous. More generally, the equality
$$\|j(\beta)-j(\alpha)\|^2=\langle\varphi,\beta^*\otimes\beta\rangle-2{\rm Re}(\langle\varphi,\alpha^*\otimes\beta\rangle)+\langle\varphi,\alpha^*\otimes\alpha\rangle$$
holds for all $\alpha, \beta\in M_c(X)$ and gives the density of  $j(E)$ whenever $E$ is a subspace of $M_c(X)$ which is dense in the weak topology. This applies in particular when $E$ is the subspace of measures with finite support and when $E$ is the subspace of the measures of the form $f\mu$, where $f\in C_c(X)$ and $\mu$ is a fixed positive Radon measure with support $X$.

For $(ii)$, we let $V$ [resp. $\underline V$] be the linear span of $e(X)$ [resp. $\underline e(X)$]. Because of the equality
$$(\sum_i\lambda_i e(x_i)| \sum_j\mu_j e(x_j))= (\sum_i\lambda_i \underline e(x_i)| \sum_j\mu_j \underline e(x_j))$$
the map $\sum_i\lambda_i e(x_i)\in V\mapsto \sum_i\lambda_i \underline e(x_i)\in\underline V$ is well defined and is isometric. It extends to an isometry $u:E\ra\underline E$. By construction, $\underline e=u\circ e$. Since this isometry is prescribed on the dense subspace $V$, it is unique.
\end{proof}

\begin{remark}\label{PT kernel} According to this theorem and its proof, a continuous kernel $\varphi$ on $X\times X$ is of positive type if and only if it satisfies one of the following equivalent properties:
\begin{itemize}
\item for all finitely supported measures $\alpha$ on $X$, $\langle\varphi,\alpha^*\otimes\alpha\rangle\,\ge 0$;
\item there exists a positive Radon measure $\mu$ such that for all $f\in C_c(X)$, $\langle\varphi,(f\mu)^*\otimes f\mu\rangle\,\ge 0$;
\item for all compactly supported measures $\alpha$ on $X$, $\langle\varphi,\alpha^*\otimes\alpha\rangle\,\ge 0$;
\item there exists a Hilbert space $E$ and a map $e:X\ra E$ such that for all $x,y\in X$, $\varphi(x,y)=(e(x)|e(y))$.
\end{itemize}

\end{remark}

\subsection{Positive type functions}

We study next continuous positive type functions on locally compact groupoids.

\begin{definition} Let $G$ be a groupoid. A complex valued function $\varphi$ defined on $G$ is said to be of positive type (abbreviated PT) if for all $x\in G^{(0)}$, the function $\varphi_x$ defined on $G^x\times G^x$ by $\varphi_x(\gamma,\gamma')=\varphi(\gamma^{-1}\gamma')$ is a positive type kernel.
\end{definition}

In the following proposition, we use the definition of a Hilbert bundle given in \cite{fd:representations}. We refer the reader to \cite{ren:bdd} for the definition of a $G$-Hilbert bundle. We use the following notation: for $x\in G^{(0)}$, $E_x$ is the fiber above $x$;  for $\gamma\in G$, $L(\gamma)$ is the linear isometry from $E_{s(\gamma)}$ to $E_{r(\gamma)}$. As said earlier, this result is well-known (see for example \cite[Theorem 1]{pat:Fourier}) but its proof is missing.

\begin{proposition}\label {PTchar} Let $(G,\lambda)$ be a locally compact groupoid with Haar system. Let $\varphi:G\ra \C$ be a continuous function. Then \tfae 
\begin{enumerate}
\item $\varphi$ is of positive type.
\item  for all $f\in C_c(G)$ and  all $x\in G^{(0)}$, one has
$$\int\int \varphi(\gamma^{-1}\gamma')\overline {f(\gamma)} f(\gamma')d\lambda^x(\gamma)d\lambda^x(\gamma')\ge 0.$$
\item there exists a continuous $G$-Hilbert bundle $E$ and a continuous section $e\in C(G^{(0)},E)$ such that
$$\varphi(\gamma)= (e\circ r(\gamma) | L(\gamma)e\circ s(\gamma)).$$
\end{enumerate}
\end{proposition}

\begin{proof} Some of the proofs are essentially the same as in the case of a group. See for example \cite[Section 13.4]{dix:C*}. Condition $(ii)$ can be written 
$$\forall x\in G^{(0)},\quad \forall f\in C_c(G),\quad\langle\varphi_x,(f\lambda^x)^*\otimes f\lambda^x\rangle\, \ge 0.$$
Therefore $(i)\Leftrightarrow (ii)$ by Remark \ref{PT kernel}. Condition $(iii)$ implies that
$$\forall x\in G^{(0)},\quad \forall \gamma,\gamma'\in G^x,\quad \varphi_x(\gamma,\gamma')=(e_x(\gamma)|e_x(\gamma')),$$
where the map $e_x: G^x\ra E_x$ is given by $e_x(\gamma)=L(\gamma)e\circ s(\gamma)$. Therefore $(iii)\Rightarrow (i)$. The novel part is $(i)\Rightarrow (iii)$. For each $x\in G^{(0)}$, the function $\varphi_x: G^x\times G^x\ra\C$ is a kernel of positive type. The GNS construction of the previous section provides a Hilbert space $E_x$ and a continuous function $e_x: G^x\ra E_x$ such that, for $(\gamma, \gamma')\in G^x\times G^x$, $\varphi_x(\gamma,\gamma')=(e_x(\gamma)|e_x(\gamma'))$ and the set $\{e_x(\gamma), \gamma\in G^x\}$ is total in $E_x$. Given $\gamma\in G_x^y$, there is a unique isometry $L(\gamma): E_x\ra E_y$ such that $L(\gamma)e_x(\gamma')=e_y(\gamma\gamma')$. Indeed, it suffices to check the equality
$$\|\sum_{i=1}^n a_i e_y(\gamma\gamma_i)\|^2=\|\sum_{i=1}^n a_i e_x(\gamma_i)\|^2$$
where $\gamma_1,\ldots,\gamma_n\in G^x$ and $a_1,\ldots,a_n\in\C$.
Both terms are equal to
$$\sum_{i,j=1}^n\overline{a_i}a_j\varphi(\gamma_i^{-1}\gamma_j).$$
Note that we can write
$$\varphi(\gamma)=(e\circ r(\gamma)| L(\gamma)e\circ s(\gamma))$$
where $e$ is the section $G^{(0)}\ra E=\coprod E_x$ defined by $e(x)=e_x(x)$. It remains to give $E$ the structure of a Hilbert bundle and to show that the action map $G*E\ra E$ is continuous. Given $f\in C_c(G)$, we define the section $j(f\lambda)$ by
$$j(f\lambda)(x)=\int e_x(\gamma)f(\gamma)d\lambda^x(\gamma).$$
The set  $\Lambda=\{j(f\lambda), f\in C_c(G)\}$ is a linear subspace of $\Pi E_x$; it is a fundamental family (i.e. for each $x$, $\Lambda(x)$ is total in $E_x$); and the norm
$$\|j(f\lambda)(x)\|=[\int \varphi(\gamma^{-1}\gamma')\overline {f(\gamma)}f(\gamma')d\lambda^x(\gamma)d\lambda^x(\gamma')]^{1/2}$$
is a continuous function of $x$. Therefore, according to \cite[Theorem 13.18]{fd:representations}, there is a unique topology on $E$ making it into a Hilbert bundle and such that all the elements of $\Lambda$ are continuous sections. The above section $e: G^{(0)}\ra E$ is continuous since for all $f\in C_c(G)$,
$$(e(x)|j(f\lambda)(x))=\int \varphi(\gamma)f(\gamma)d\lambda^x(\gamma)$$
is a continuous function of $x$. It remains to show that the action map $G*E\ra E$ is continuous. This amounts to showing the continuity of the bundle map $a: s^*E\ra r^*E$ sending $(\gamma, \xi)\in G\times E_{s(\gamma)}$ to $(\gamma, L(\gamma)\xi)\in G\times E_{r(\gamma)}$. According to \cite[Proposition 13.16]{fd:representations}, this amounts to showing that for all $\xi$ in a fundamental family of sections, the section $a\circ \xi$ is continuous. The pullback bundle $s^*E$ has the fundamental family of sections $C_c(G)\otimes\Lambda$, where for $g\in C_c(G)$ and $\xi\in\Lambda$, $(g\otimes\xi)(\gamma)=g(\gamma)\xi\circ s(\gamma)$. Let $\xi=g\otimes j(f\lambda)\in C_c(G)\otimes\Lambda$, where $f,g\in C_c(G)$. Then
$$\begin{array}{ccl}
a\circ\xi(\gamma)&=& g(\gamma)L(\gamma)j(f\lambda)(s(\gamma))\\
&=& \int L(\gamma)L(\gamma')e\circ s(\gamma')g(\gamma)f(\gamma')d\lambda^{s(\gamma)}(\gamma')\\
&=&\int L(\gamma\gamma')e\circ s(\gamma')g(\gamma)f(\gamma')d\lambda^{s(\gamma)}(\gamma')\\
&=&\int L(\gamma')e\circ s(\gamma')g(\gamma)f(\gamma^{-1}\gamma')d\lambda^{r(\gamma)}(\gamma')\\
&=&\int e(\gamma')g(\gamma)f(\gamma^{-1}\gamma')d\lambda^{r(\gamma)}(\gamma')\\
\end{array}$$
It is of the form $\gamma\mapsto \int e(\gamma')F(\gamma,\gamma')d\lambda^{r(\gamma)}(\gamma')$ with $F\in C_c(G*_rG)$, where $G*_rG=\{(\gamma,\gamma')\in G\times G: r(\gamma)=r(\gamma')\}$. Such a section is continuous. Indeed, this is true when $F$ is of the form $F(\gamma,\gamma')=f(\gamma)g(\gamma')$, where $f,g\in C_c(G)$, because it is of the form $f\otimes j(g\lambda)$. This is still true for an arbitrary $F\in C_c(G*_rG)$ because it can be be approximated uniformly by linear combinations of such functions of that form (see \cite[Proposition 14.1]{fd:representations}).

\end{proof}

We can give the main result about continuous positive type functions, which generalizes above Theorem \ref{kernel}.

\begin{theorem}\label{PTfunction}
Let $G$ be a locally compact groupoid with Haar system and let
$\varphi$ be a continuous positive type function on $G$. Then,
\begin{enumerate}
\item there exists a continuous $G$-Hilbert bundle $E$ and a continuous map $e:G^{(0)}\ra E$ such that
\begin{enumerate}
\item for all $\gamma\in G$, $\varphi(\gamma)=(e\circ r(\gamma)|L(\gamma) e\circ s(\gamma))$;
\item for all $x\in G^{(0)}$, the linear span of $\{L(\gamma)e\circ s(\gamma), \gamma\in G^x\}$ is dense in $E_x$.
\end{enumerate}
\item the pair $(E,e)$ is unique, in the sense that if $(E',e')$ also satisfies the conditions of $(i)$, there exists a unique isomorphism of $G$-Hilbert bundles 
$u:E\ra E'$ such that $e'=u\circ e$.
\end{enumerate}
\end{theorem}

\begin{proof}
The existence of $E$ and $e$ is given above. Note that for $\gamma\in G^x$, $L(\gamma)e\circ s(\gamma)=e_x(\gamma)$. By construction, for all $x\in G^{(0)}$,
$\{e_x(\gamma), \gamma\in G^x\}$ is total in $E_x$.
\sk
Let us show the uniqueness of $(E,e)$ as stated in $(ii)$. Let $(E',e')$ be a pair  satisfying $(a)$ and $(b)$. For each $x\in G^{(0)}$, according to Theorem \ref{kernel} there exists a unique isometry $u_x:E_x\ra E'_x$ such that $e'_x=u_x\circ e_x$, where $e'_x: G^x\ra E'_x$ is defined by $e'_x(\gamma)=L'(\gamma)e'\circ s(\gamma)$. This gives a bundle map $u:E\ra E'$ which is $G$-equivariant. In order to show that it is continuous, it suffices to check that for all sections $\xi$ in a fundamental family of continuous sections of $E$, $u\circ\xi$ is a continuous section of $E'$. Let us consider $\xi=j(f\lambda)$, where $f\in C_c(G)$. Then $u\circ\xi=j'(f\lambda)$, where
$$j'(f\lambda)(x)=\int L'(\gamma)e'\circ s(\gamma)f(\gamma)d\lambda^x(\gamma)\in E'_x.$$
Let us show that this section of $E'$  is  continuous. The map
$$\gamma\mapsto L'(\gamma)e'\circ s(\gamma)f(\gamma)$$
is a continuous section of the pullback bundle $r^*E'$ and it has compact support.
The conclusion is given by the following lemma.
\end{proof}

\begin{lemma} Let $\xi$ be a continuous section with compact support of a pull-back bundle $\pi^*E$, where $p:X\ra Y$ is continuous and open and where $E$ is a Banach bundle over $Y$. Let $\alpha$ be a continuous $p$-system. Then $\alpha(\xi)(y)=\int \xi(x)d\alpha_y(x)$ is a continuous section of $E$.
\end{lemma}

\begin{proof}
This is true for a section of the form $\xi=f\otimes\eta$, where $f\in C_c(X)$ and $\eta\in C(Y,E)$. This remains true for an arbitrary continuous section with compact support because it can be uniformly approximated by linear combinations of such sections.
\end{proof}

\subsection{Conditionally negative type kernels}
Again, the material of this section is well known and covered in \cite{bhv:T}. Since the general groupoid case relies on this particular case, we recall it here.
\begin{definition}
Let $X$ be a set. A function $\psi: X\times X\ra \R$ is said to be conditionally of negative type (CNT for short) if 
\begin{enumerate} 
\item $\forall x\in X, \psi(x,x)=0$, 
\item $\forall (x,y)\in X\times X, \psi(y,x)=\psi(x,y)$ and
\item $\forall n\in\N^*$, $\forall \zeta_1,\ldots,\zeta_n\in\R$ such that $\sum_1^n\zeta_i=0$ and $\forall x_1,\ldots,x_n\in X,$
$$\sum_{i,j}\psi(x_i,x_j)\zeta_i\zeta_j\le 0.
$$
\end{enumerate}
\end{definition} 
Just as the scalar product of a complex Hilbert space is the model of a positive type kernel, the square of the distance function of a real affine Hilbert space is the model of a conditionally negative type kernel. Let us fix our notation. We write an affine space as a pair $(E,\vec E)$, where the vector space $\vec E$ acts on the space $E$ on the right freely and transitively. When we speak of an affine space $E$, the direction space $\vec E$ is implicit. In this section, we consider real vector spaces. Given $x,y\in E$, we write $\overrightarrow{xy}=y-x\in\vec E$. We say that an affine space $(E,\vec E)$ is a real affine Hilbert space if $\vec E$ is a real Hilbert space. Given a locally compact space $X$, we write $C(X,\R)$, $M(X,\R)$, etc. for the space of real continuous functions, real measures, etc.

\begin{proposition}\label{CNTkernel} Let $X$ be a locally compact Hausdorff space. Let $\psi: X\times X\ra \R$ be a continuous function such that $\psi(x,x)=0$ for all $x\in X$ and $\psi(y,x)=\psi(x,y)$ for all $(x,y)\in X\times X$. Then \tfae 
\item{(i)} $\psi$ is conditionally of negative type;
\item{(ii)} if $\mu$ is a fixed positive Radon measure with full support,  for every $f\in C_c(G,\R)$ such that $\int f d\mu=0$, one has
$\int\int \varphi(x,y)f(x)f(y)d\lambda(x)d\lambda(y)\le 0$;
\item{(iii)}  for $\alpha\in M_c(X,\R)$ such that $\alpha(X)=0$, one has
$\int\int \varphi(x,y)d\alpha(x)d\alpha(y)\le 0$.
\end{proposition}
\begin{proof} It is essentially the same density argument as the one we used in 2.1. Let $M_c^0(X,\R)=\{\alpha\in M_c(X,\R): \alpha(X)=0\}$. Let $E$ be either the subspace of finitely supported measures or the subspace of measures of the form $f\mu$ where $f\in C_c(X,\R)$. Then $E\cap M^0_c(X,\R)$ is dense in $M^0_c(X,\R)$ in the weak topology. This gives the equivalences $(i)\Leftrightarrow (iii)$ and $(ii)\Leftrightarrow (iii)$.
\end{proof}
The next theorem is \cite[Theorem C.2.3]{bhv:T}. In fact, as shown there, it holds for an arbitrary topological space $X$. We include a proof for completeness.
\begin{theorem} Let $X$ be a locally compact Hausdorff space. Let $\psi: X\times X\ra \R$ be a continuous function conditionally of negative type. Then 
\begin{enumerate}
\item There exists a pair $(E,e)$ consisting of a real affine Hilbert space $E$ and a continuous function $e:X\ra  E$ such that
\begin{enumerate}
\item for all $(x,y)\in X\times X$, $\psi(x,y)=\|e(x)-e(y)\|^2$;
\item $E$ is the closed affine subspace generated by $e(X)$.
\end{enumerate}
\item If $(E',e')$ is another pair as in $(i)$, there exists a unique affine isometry $u:E\ra E'$ such that $e'=u\circ e$.
\end{enumerate}
\end{theorem}

\begin{proof} To construct $(E,c)$, we proceed exactly as in the case of a  function of positive type.
We introduce the affine subspace of $M_c(X,\R)$:
$$M^1_c(X,\R)=\{\alpha\,\hbox{compactly supported signed Radon measures on}\, X : \alpha(X)=1\}.$$
Its associated vector space is
$M^0_c(X,\R)$. We define a pre-inner product on $M^0_c(X,\R)$ and a quasi-metric (i.e. satisfying the axioms of a metric except the separation property) by
$$(\alpha,\beta)\mapsto -<\psi/2,\alpha\otimes\beta>\quad{\rm for}\quad\alpha,\beta\in M^0_c(X,\R),$$
$$(\alpha,\beta)\mapsto (-<\psi/2,\alpha-\beta\otimes\alpha-\beta>)^{1/2}\quad{\rm for}\quad\alpha,\beta\in M^1_c(X,\R).$$
We let $\vec E$ be the Hilbert space obtained by separating and completing $M^0_c(X,\R)$ with respect to the associated semi-norm. The separation/completion of $M^1_c(X,\R)$ is an affine real Hilbert space $E$ having $\vec E$ as its associated Hilbert space. We denote by $j:M^1_c(X,\R)\ra E$ the canonical map. It is an affine map having the canonical map $\vec j: M^0_c(X,\R)\ra \vec E$ as its associated linear map. For $\alpha, \beta\in M^0_c(X,\R)$, we have
$$(\vec j(\alpha)| \vec j(\beta))=-<\psi/2,\alpha\otimes\beta>.$$
For $\alpha, \beta\in M^1_c(X,\R)$, $\|j(\alpha)-j(\beta)\|^2$ is given by:
$$-(1/2)\int \psi(x,y)d\alpha(x)d\alpha(y)+\int\psi(x,y)d\alpha(x)d\beta(y) -(1/2)\int\psi(x,y)d\beta(x)d\beta(y).$$
We define $e:X\ra E$ by $e(x)=j(\epsilon_x)$. Then
$$\|e(x)-e(y)\|^2=\psi(x,y).$$
The continuity of $e$ is obvious. For the assertion $(b)$, we let $E\subset M_c(X,\R)$ be the subspace of finitely supported measures. As we have seen, $E^1=E\cap M_c^1(X,\R)$ is dense in $M_c^1(X,\R)$ in the weak topology. One deduces from the above formula that $\overline{j(E^1)}=E$. We are done since $j(E^1)$ is exactly the affine subspace generated by $e(X)$.

Let us show uniqueness. The map $e$ [resp. $e'$] extends unique to an affine map $j:E^1\ra E$ [resp. $j':E^1\ra E'$] such that 
$j(\sum_{i=1}^n\lambda_i\epsilon_{x_i})=\sum_{i=1}^n\lambda_i e(x_i)$ 
for all $n\in\N^*$, all $x_1,\ldots,x_n\in X$ and all $\lambda_1,\ldots,\lambda_n\in\R$ such that $\sum_{I=1}^n\lambda_1=1$. The above formula shows that for all $\alpha,\beta\in E^1$, $\|j(\alpha)-j(\beta)\|=\|j'(\alpha)-j'(\beta)\|$. Therefore, there is an isometry $u:j(E^1)\ra j'(E^1)$ sending $j(\alpha)$ to $j'(\alpha)$. This isometry extends by uniform continuity to an isometry from $E$ onto $E'$. As a continuous affine map, such an isometry is completely determined by its values on $e(X)$; therefore it is unique.
\end{proof}

We shall need the following result, which is an immediate consequence of the above. Given a continuous function $e:X\ra  E$, where $X$ is a locally compact Hausdorff space and $E$ is a real affine Hilbert space, we define $j(\alpha)=\int e(x)d\alpha(x)$ for every $\alpha\in M_c^1(X,\R)$. Then we have

\begin{proposition}\label{density} Let $\mu$ be a positive Radon measure on $X$ with full support. Then $\{j(f\lambda): f\in C_c(X,\R), \int fd\mu=1\}$ and the affine subspace generated by $e(X)$ have the same closure in $E$.
\end{proposition}

\subsection{Conditionally negative type functions}

Let us turn now to conditionally negative type functions on locally compact groupoids.

\begin{definition} Let $G$ be a groupoid. A real valued function $\psi$ defined on $G$ is said to be conditionally of negative type (CNT for short) if for all $x\in G^{(0)}$, the function $\psi_x$ defined on $G^x\times G^x$ by $\psi_x(\gamma,\gamma')=\psi(\gamma^{-1}\gamma')$ is a conditionally negative type kernel.
\end{definition}

Note that this implies that $\psi$ vanishes on $G^{(0)}$ and $\psi$ is symmetric.

In the following proposition, we use the notion of affine real Hilbert bundle. Given a Hilbert bundle $\vec E$ over a topological space $X$, an affine Hilbert bundle is a topological bundle $E$ over $X$ on which $\vec E$ acts on the right freely, transitively and properly. In other words we have an action map

$$\begin{array}{cc}
\Phi:E*\vec E& \ra  E\\
(u,\vec a)&\mapsto u+\vec a\\
\end{array}$$
such that the map:
$$\begin{array}{cc}
E*\vec E& \ra  E*E\\
(u,\vec a)&\mapsto (u,u+\vec a)\\
\end{array}$$
is a homeomorphism. A $G$-affine Hilbert bundle is an affine Hilbert bundle endowed with a continuous action $G*E\ra E$ such that each $\gamma\in G$ defines an affine isometry $A(\gamma): E_{s(\gamma)}\ra E_{r(\gamma)}$. Then, we denote by $L(\gamma):\vec E_{s(\gamma)}\ra \vec E_{r(\gamma)}$ the associated  linear isometric action. We shall write $\gamma u$ instead of $A(\gamma)(u)$ when it does not lead to a confusion.

We shall need the following lemma.

\begin{lemma}\label {continuous section} Let $\xi$ be a continuous section of a pull-back bundle $p^*E$, where $p:X\ra Y$ is continuous and open and where $E$ is an affine Banach bundle over $Y$. Let $\alpha=(\alpha_y)_{y\in Y}$ be a continuous $p$-system, where for all $y\in Y$, $\alpha_y$ belongs to $M_c(p^{-1}(y))$ and $\alpha_y(p^{-1}(y))=1$. Then $\alpha(\xi)(y)=\int \xi(x)d\alpha_y(x)$ is a continuous section of $E$.
\end{lemma}

\begin{proof}
This is true for a section of the form $\xi(x)=\sum_{i=1}^nf_i(x)\eta_i(p(x))$, where $f_1,\ldots,f_n$ belong to  $C(X,\R)$, satisfy $\sum_{i=1}^n f_i(x)=1$ for all $x\in X$ and $\eta_1,\ldots,\eta_n$ are continuous sections of $E\ra Y$. This remains true for an arbitrary continuous section because it can be approximated uniformly on compact subsets by such sections (see \cite[Proposition 14.1]{fd:representations}).

\end{proof}

\begin{proposition}\label{CNTchar} (cf. \cite[section 3.3]{tu:conjecture}) Let $(G,\lambda)$ be a locally compact groupoid with Haar system. Let $\psi: G\ra \R$ be a  continuous function. Then \tfae 
\begin{enumerate}
\item $\psi$ is conditionally of negative type;
\item $\psi$ vanishes on $G^{(0)}$, is symmetric and for all $x\in G^{(0)}$ and all $f\in C_c(G,\R)$ such that $\int f d\lambda^x=0$, we have
$$\int\int \varphi(\gamma^{-1}\gamma')f(\gamma) f(\gamma')d\lambda^x(\gamma)d\lambda^x(\gamma')\le 0;$$
\item there exists a continuous $G$- affine real Hilbert bundle $E$ and a continuous section $e\in C(G^{(0)},E)$ such that
$$\psi(\gamma)=\|e\circ r(\gamma)- \gamma e\circ s(\gamma)\|^2.$$
\end{enumerate}
\end{proposition}

\begin{proof} The equivalence $(i)\Leftrightarrow (ii)$ is a consequence of Proposition \ref{CNTkernel}. It also uses the fact that if $f\in C_c(G,\R)$, $f_{|G^x}\in C_c(G^x,\R)$ and that for every $g\in C_c(G^x,\R)$, there exists $f\in C_c(G,\R)$ such that $f_{|G^x}=g$. The implication $(iii)\Rightarrow (i)$ is easy: if $(iii)$ holds, we can write 
$\psi_x(\gamma,\gamma')=\|e_x(\gamma)-e_x(\gamma')\|^2$ where $e_x(\gamma)=\gamma e\circ s(\gamma)$ for $\gamma\in G^x$; therefore $\psi_x$ is conditionally of negative type. Let us show the implication $(ii)\Rightarrow (iii)$. The construction of the affine Hilbert bundle $E$ is similar to the construction of the Hilbert bundle in Proposition \ref{PTchar}. For each $x\in G^{(0)}$, the function $\psi_x: G^x\times G^x\ra\R$ is a kernel conditionally of negative type. The GNS construction of the previous section provides an affine real Hilbert space $E_x$ and a continuous function $e_x: G^x\ra E_x$ such that, for $(\gamma, \gamma')\in G^x\times G^x$, $\psi_x(\gamma,\gamma')=\|e_x(\gamma)-e_x(\gamma')\|^2$ and the set $\{e_x(\gamma), \gamma\in G^x\}$ is total in $E_x$. Let $\gamma\in G_x^y$. The pair $(E_y,e')$, where $e': G^x\ra E_y$ is defined by $e'(\gamma')=e_y(\gamma\gamma')$ satisfies the conditions $(a)$ and $(b)$ of Theorem \ref{CNTkernel}. Therefore, there is a unique affine isometry $A(\gamma): E_x\ra E_y$ such that $A(\gamma)e_x(\gamma')=e_y(\gamma\gamma')$. 
By construction we can write
$$\psi(\gamma)=\|e\circ r(\gamma)-A(\gamma)e\circ s(\gamma)\|^2$$
where $e$ is the section $G^{(0)}\ra E=\coprod E_x$ defined by $e(x)=e_x(x)$. The next step is to endow $E=\coprod E_x$ and $\vec E=\coprod \vec E_x$ with topologies turning $(E,\vec E)$ into an affine real Hilbert bundle. As usual, this is done by giving fundamental families of continuous sections. Let us say that a closed subset $A\subset G$ is $r$-compact if for each compact subset $K\subset G^{(0)}$, $A\cap r^{-1}(K)$ is compact. We denote by ${\mathcal E}$ the real affine space of continuous real-valued functions $f$ on $G$ with $r$-compact support such that for all $ x\in G^{(0)}$, $\int f d\lambda^x=1$.  Its direction space $\vec{\mathcal E}$ is the linear space of fucntions $f$ such that for all $ x\in G^{(0)}$, $\int f d\lambda^x=0$. Given $f\in{\mathcal E}$, we define the section of $E$:
$$j(f\lambda)(x)=j_x(f\lambda^x)=\int e_x(\gamma)f(\gamma)d\lambda^x(\gamma).$$
Consider the affine space of sections $\Lambda=\{j(f\lambda): f\in {\mathcal E}\}$. Let us check that it satisfies the conditions $(a)$ and $(b)$ of \cite[Theorem 13.18]{fd:representations} (or rather its affine version). Condition $(a)$ is the continuity of the distance function. Given $f,g\in{\mathcal E}$, the distance function $x\mapsto \|j(f\lambda)(x)-j(g\lambda)(x)\|$ is continuous since:
$$\|j(f\lambda)-j(g\lambda)\|_x^2=-{1\over 2}\int\int \psi(\gamma^{-1}\gamma')(f-g)(\gamma)(f-g)(\gamma')d\lambda^x(\gamma)d\lambda^x(\gamma').$$
Condition $(b)$ is the density, for each $x\in G^{(0)}$ of $\Lambda(x)=\{j_x(f\lambda^x): f\in {\mathcal E}\}$ in $E_x$. We know that $\{j_x(f_x\lambda^x): f_x\in C_c(G^x,\R), \int f_xd\lambda^x=1\}$ is dense in $E_x$. But this space is $\Lambda(x)$: by using a partition of unity, one can extend $f_x\in C_c(G^x,\R)$ such that $\int f_xd\lambda^x=1$ to $f\in{\mathcal E}$. The topology of $E$ is defined by the following basis of open sets: we fix $\xi=j(f\lambda)\in\Lambda$, $U$ open subset of $X$ and $\epsilon>0$ and set
$$W(\xi, U,\epsilon)=\{u\in E: \pi(u)\in U, \|u-\xi\circ\pi(u)\|<\epsilon\}.$$
where $\pi:E\ra X$ is the projection. By using similarly $\vec{\mathcal E}$ and $\vec\Lambda=\{\vec j(f\lambda): f\in \vec{\mathcal E}\}$, we endow $\vec E=\coprod \vec E_x$ with a topology of real Hilbert bundle. This turns $(E,\vec E)$ into a real affine Hilbert bundle. Next, we have to check the continuity of the section $e$. It suffices to check the continuity, for all $f\in{\mathcal E}$ of the function
$x\mapsto \|e(x)-j(f\lambda)(x)\|^2$; this is clear because this last expression is just:
$$\int\psi(\gamma)f(\gamma)d\lambda^x(\gamma)-{1\over 2}\int\psi(\gamma^{-1}\gamma')f(\gamma)f(\gamma')d\lambda^x(\gamma)d\lambda^x(\gamma').$$
It remains to check the continuity of the action map $(\gamma,u)\in G*E\mapsto A(\gamma)u\in E$. We proceed as in the proof of Proposition \ref{PTchar}: we prove the continuity of the bundle map $a: s^*E\ra r^*E$ sending $(\gamma, u)\in G\times E_{s(\gamma)}$ to $(\gamma, A(\gamma)u)\in G\times E_{r(\gamma)}$ by checking that it sends the sections of a fundamental family of continuous sections of $s^*E$ to continuous sections of $r^*E$. Let us take a section $\xi$ of $s^*E$ of the form
$$\xi(\gamma)=\int e(\gamma')f(\gamma,\gamma')d\lambda^{s(\gamma)}(\gamma')$$
where $f$ is continuous real-valued function  on $G_s\times_rG=\{(\gamma,\gamma')\in G\times G: s(\gamma)=r(\gamma')\}$ with $p_1$-compact support, where $p_1:G_s\times_rG\ra G$ is the first projection, and such that for all $\gamma\in G$, $\int f(\gamma,\gamma')d\lambda^{s(\gamma)}(\gamma')=1$. According to Lemma \ref{continuous section}, it is continuous; moreover these sections form a fundamental family of continuous sections of $s^*E$. Then
$$a\circ\xi(\gamma)=\int e(\gamma\gamma')f(\gamma,\gamma')d\lambda^{s(\gamma)}(\gamma')=\int e(\gamma')g(\gamma,\gamma')d\lambda^{r(\gamma)}(\gamma')$$
where $g$ is the function defined on $G_r\times_rG$ by $g(\gamma,\gamma')=f(\gamma,\gamma^{-1}\gamma')$. It is continuous, it has a support $p_1$-compact and satisfies $\int f(\gamma,\gamma')d\lambda^{r(\gamma)}(\gamma')=1$ for all $\gamma\in G$. By the same token, $a\circ\xi$ is a continuous section or $r^*E$.
\end{proof}
 
 We conclude by the main result about representation of continuous conditionally negative type functions on groupoids which generalizes Theorem \ref{CNTkernel}.

\begin{theorem}\label{CNTfunction}
Let $G$ be a locally compact groupoid with Haar system and let
$\psi$ be a real continuous function on $G$ conditionally of negative type. Then,
\begin{enumerate}
\item there exists a continuous $G$-real affine Hilbert bundle $E$ and a continuous map $e:G^{(0)}\ra E$ such that
\begin{enumerate}
\item for all $\gamma\in G$, $\psi(\gamma)=\|e\circ r(\gamma)- A(\gamma) e\circ s(\gamma)\|^2$.
\item for all $x\in G^{(0)}$, the affine span of $\{A(\gamma) e\circ s(\gamma), \gamma\in G^x\}$ is dense in $E_x$.
\end{enumerate}
\item the pair $(E,e)$ is unique, in the sense that if $(E',e')$ also satisfies the conditions of $(i)$, there exists a unique isomorphism of $G$-real affine Hilbert bundles 
$u:E\ra E'$ such that $e'=u\circ e$.
\end{enumerate}
\end{theorem}

\begin{proof}
The existence of $E$ and $e$ is given above. Note that for $\gamma\in G^x$, $A(\gamma)e\circ s(\gamma)=e_x(\gamma)$. By construction, for all $x\in G^{(0)}$,
$\{e_x(\gamma), \gamma\in G^x\}$ is total in $E_x$.
\sk
Let us show the uniqueness of $(E,e)$ as stated in $(ii)$. Let $(E',e')$ be a pair  satisfying $(a)$ and $(b)$. For each $x\in G^{(0)}$, according to Theorem \ref{CNTkernel} there exists a unique affine isometry $u_x:E_x\ra E'_x$ such that $e'_x=u_x\circ e_x$, where $e'_x: G^x\ra E'_x$ is defined by $e'_x(\gamma)=L'(\gamma)e'\circ s(\gamma)$. This gives a bundle map $u:E\ra E'$ which is $G$-equivariant. In order to show that it is continuous, it suffices to check that for all sections $\xi$ in a fundamental family of continuous sections of $E$, $u\circ\xi$ is a continuous section of $E'$. Let us consider $\xi=j(f\lambda)$, where $f$ is a continuous real function on $G$ with $r$-compact support and such that $\int f d\lambda^x=1$ for all $x\in G^{(0)}$. Then $u\circ\xi=j'(f\lambda)$, where
$$j'(f\lambda)(x)=\int A'(\gamma)e'\circ s(\gamma)f(\gamma)d\lambda^x(\gamma)\in E'_x.$$
Let us show that this section of $E'$  is  continuous. The map
$$\gamma\mapsto A'(\gamma)e'\circ s(\gamma)$$
is a continuous section of the pull-back bundle $r^*E'$. The conclusion is given by Lemma \ref{continuous section}.
\end{proof}

The above theorem can be expressed in terms of linear representations and cocycles.

\begin{definition} Let $\vec E$ be a $G$-vector bundle. A (one-)cocycle is a section $c: G\ra r^*\vec E$ (i.e. $c(\gamma)\in \vec E_{r(\gamma)}$) such that $c(\gamma\gamma')=c(\gamma)+L(\gamma)c(\gamma')$ where $L$ denotes the linear action of $G$. It is a coboundary if there exists a section $\xi: X\ra \vec E$ such that $c(\gamma)=\xi\circ r(\gamma)-L(\gamma)\xi\circ s(\gamma)$.
\end{definition}

In our case, $G$ is a topological groupoid, $\vec E$ is a a $G$-Hilbert bundle and we shall only consider continuous cocycles and continuous sections.
The relation between cocycles and affine bundles is well-known: let $E$ be a $G$-affine Hilbert bundle. A continuous section $e: G^{(0)}\ra E$ identifies $E$ with $\vec E$ (as topological spaces) and defines a continuous section $c: G\ra r^*\vec E$ (where $r^*\vec E$ is the pull-back along $r: G\ra G^{(0)}$) according to:
$$c(\gamma)=A(\gamma)e\circ s(\gamma)-e\circ r(\gamma).$$
$c$ satisfies the cocycle property:
$$c(\gamma\gamma')=c(\gamma)+L(\gamma)c(\gamma'),$$
where $L$ is the linear part of the affine action $A$.
Another choice of section gives a cohomologous cocycle.
Conversely, a continuous cocycle $c:G\ra r^*\vec E$ endows $E=\vec E$ with a $G$-affine action according to:
$A(\gamma)e=c(\gamma)+L(\gamma)e$. The $G$-affine Hilbert bundle defined by $c$ will be denoted by $E(c)$.

Here is another formulation of the theorem (cf. \cite[section 3.3]{tu:conjecture}).

\begin{theorem}\label{GNS} Let $(G,\lambda)$ be a locally compact groupoid with Haar system. Let $\psi: G\ra \R$ be a  continuous function conditionally of negative type. Then 
\begin {enumerate}
\item There exists a pair $(\vec E, c)$ consisting of a continuous $G$-Hilbert bundle $\vec E$ and a continuous cocycle $c:G\ra r^*\vec E$ such that
\begin{itemize}
\item for all $\gamma\in G$, $\psi(\gamma)=\|c(\gamma)\|^2$;
\item for all $x\in G^{(0)}$, $\{c(\gamma), \gamma\in G^x\}$ is total in $\vec E_x$.
\end{itemize}
\item If $(\vec E', c')$ is another pair satisfying the same properties, there exists a $G$-equivariant linear isometric continuous bundle map $u: \vec E\ra \vec E'$ such that $c'=u\circ c$.
\end{enumerate}
\end{theorem}

\section{Correspondences}

\subsection{Correspondences} Let us recall the definition of a correspondence in the C*-algebraic framework. A standard reference for Hilbert C*-modules (abbreviated here to C*-modules) is \cite{lan:C*}. Our C*-modules will be right C*-modules. Given a C$^*$-module $\mathcal E$ over the C*-algebra $B$, ${\mathcal L}_B({\mathcal E})$ is the C*-algebra of bounded adjointable $B$-linear maps and ${\mathcal K}_B({\mathcal E})$ is the ideal of ``compact'' operators.
\begin{definition}\label{C*-correspondence}  Let $A$ and
$B$ be  C$^*$-algebras. An {\it $(A,B)$-C*-correspondence} is a right
$B$-C$^*$-module
$\mathcal E$ together with a $*$-homomorphism $\pi:A\rightarrow {\mathcal L}_B({\mathcal E})$.
\end{definition}

We shall usually view an $(A,B)$-C*-correspondence as an $(A,B)$-bimodule. There is a composition of correspondences: given a C*-correspondence $\mathcal E$ from $A$ to $B$ and a
C*-correspondence $\mathcal F$ from $B$ to $C$, one can construct the C*-correspondence
${\mathcal E}\otimes_B{\mathcal F}$ from $A$ to $C$. It is the $C$-C$^*$-module obtained by separation and
completion of the ordinary tensor product ${\mathcal E}\otimes{\mathcal F}$ with respect to the inner product
$$<x\otimes y,x'\otimes y'>=<y,<x,x'>_B y'>_C,\qquad x,x'\in {\mathcal E}\quad y,y'\in {\mathcal F};$$
the left $A$ action is given by $a(x\otimes y)=ax\otimes y$ for $a\in A$.

\subsection{G-Hilbert bundles and correspondences} Let $G$ be a locally compact groupoid endowed with a Haar system. We are going to define a covariant functor from the category of $G$-Hilbert bundles to the category of $(A,A)$-correspondences, where $A$ is either the full or the reduced C*-algebra of $G$. Moreover, this functor transforms the tensor product of $G$-Hilbert bundles into the composition of correspondences. This functor is well-known in the case when $G=X$ is a space, where it generalizes the construction of a projective finitely generated module from a vector bundle, and in the case when $G$ is a group, where it motivates the definition of propery (T) for von Neumann algebras (see \cite{cj:T}). The construction for locally compact groupoids with Haar system is given in Section 7.2 of P.-Y. Le Gall's Ph.D. thesis \cite{leg:these}, in the more general context of groupoid crossed products, as part of the descent functor from the equivariant KK-theory to the KK-theory of the crossed products. Since this reference is not readily available, it may be useful to  recall this construction here. We shall give a more pedestrian presentation and relate it to the elegant construction of \cite{leg:these}.

The data are a groupoid with Haar system $(G,\lambda)$ and a $G$-Hilbert bundle $E$. The $*$-algebra $C_c(G)$ is defined as usual (see \cite{ren:approach}). We denote by $C_c(G,r^*E)$ the space of continuous sections of the pull-back bundle $r^*E$.

We define the left and right actions of $C_c(G)$ on $C_c(G,r^*E)$: for $f,g\in C_c(G)$ and $\xi\in C_c(G,r^*E)$,
$$f\xi(\gamma)=\int f(\gamma')L(\gamma')\xi({\gamma'}^{-1}\gamma)d\lambda^{r(\gamma)}(\gamma')$$
$$\xi g(\gamma)=\int \xi(\gamma\gamma') g({\gamma'}^{-1})d\lambda^{s(\gamma)}(\gamma')$$
and the right inner product: for $\xi, \eta\in C_c(G)$,
$$<\xi,\eta>(\gamma)= \int (\xi({\gamma'}^{-1})|\eta({\gamma'}^{-1}\gamma))_{s(\gamma')}d\lambda^{r(\gamma)}(\gamma').$$
We will complete $C_c(G,r^*E)$ into a C*-correspondence in the above sense.
We define ${\mathcal E}=C^*(G,r^*E)$ [resp. $C_r^*(G,r^*E)$] as the completion of $C_c(G,r^*E)$ with respect to the norm $\|\xi\|=\|<\xi,\xi>\|_A$ where $A=C^*(G)$ [resp. $C_r^*(G)$]. We shall check that the left action of $C_c(G)$ extends to a $*$-homomorphism $\pi:A\rightarrow {\mathcal L}_A({\mathcal E})$. Let ${\mathcal F}=C_0(G^{(0)},E)$ denote the space of continuous sections of $E$ vanishing at infinity; it  is a right C*-module over $B=C_0(G^{(0)})$.  On the other hand, $A$ is naturally a $(B,A)$-correspondence since $B$ is a subalgebra of the multiplier algebra of $A$. Le Gall associates to ${\mathcal F}$ the right $A$-C*-module ${\mathcal F}\otimes_B A$. 

\begin{lemma}\label{Le Gall} Let $(G,\lambda)$ be a locally compact groupoid with Haar system and let $E$ be a $G$-Hilbert bundle. Then
\begin{enumerate}
\item $\,{\mathcal E}=C^*(G,r^*E)$  is isomorphic as a right C*-module over $A=C^*(G)$ to  ${\mathcal F}\otimes_B A$, where ${\mathcal F}=C_0(G^{(0)},E)$ and $B=C_0(G^{(0)})$.
\item a similar result holds for $\,{\mathcal E}=C_r^*(G,r^*E)$ and $A=C_r^*(G)$.
\end{enumerate}
\end{lemma}

\begin{proof}
We consider the case $A=C^*(G)$. The case $A=C_r^*(G)$  is similar. We define a map $j$ from the algebraic tensor product $C_c(G^{(0)},E)\otimes C_c(G)$ to $C_c(G,r^*E)$  by $j(\xi\otimes f)(\gamma)=\xi\circ r(\gamma) f(\gamma)$, where $\xi\in C_c(X,E)$, $f\in C_c(G)$ and $\gamma\in G$. Recall that the inner product on $C_c(G^{(0)},E)\otimes C_c(G)$ is defined by
$$<\xi\otimes f, \eta\otimes g>=<f,<\xi,\eta>_Bg>,\qquad \xi,\eta\in C_c(X,E)\quad f,g\in C_c(G).$$
It is easily checked that, for $\xi, \eta\in C_c(G^{(0)},E)\otimes C_c(G)$ and $g\in C_c(G)$, we have
$<j(\xi), j(\eta)>=<\xi,\eta>$ and $j(\xi g)=j(\xi)g$. Therefore $j$ extends to a C*-module isomorphism $j:{\mathcal F}\otimes_B A\ra C^*(G,r^*E)$. 
\end{proof}

\begin{proposition} The left action of $C_c(G)$ on $C_c(G,r^*E)$ extends to a non\-degenerate representation of $C^*(G)$ on $\mathcal E$ (resp. $C_r^*(G)$ on ${\mathcal E}_r$) by bounded adjointable operators.
\end{proposition}

\begin{proof} We first assume that $A=C^*(G)$. Our reference for representations of groupoids and their convolution algebras are \cite{ren:approach} and \cite{ren:representations}. It is clearer if we distinguish $\underline A=A$ acting on the left from $A$ acting on the right.
We first prove the inequality $\|f\xi\|\le \|f\|_{\underline A}\|\xi\|$ for $f\in C_c(G)$ and $\xi\in C_c(G,r^*E)$. Let $(\mu,H)$ be a representation of $G$: $\mu$ is a quasi-invariant measure and $H$ is a measurable $G$-Hilbert bundle. It defines the integrated representation $\pi$ of $A$ on $L^2(G^{(0)},\mu,H)$. Let us construct the representation $\pi'$ of $C_c(G)$ induced by $\pi$ through the bimodule $C_c(G,r^*E)$. I claim that it is unitarily equivalent to  the integrated representation $\underline\pi$ of $(\mu,E\otimes H)$, where $\gamma(e\otimes\xi)=L(\gamma)e\otimes\pi(\gamma)e$. Indeed, we first check that the linear map
$$U:C_c(G,r^*E)\otimes L^2(G^{(0)},H)\ra L^2(G^{(0)}, E\otimes H)$$ defined for $e\in C_c(G,r^*E)$ and $\xi\in L^2(G^{(0)},H)$ by
$$U(e\otimes\xi)(x)=\int e(\gamma)\otimes\pi(\gamma)\xi\circ s(\gamma)\delta^{-1/2}(\gamma)d\lambda^x(\gamma)$$
where $\delta$ is the Radon-Nikodym cocycle of $\mu$, satisfies
$$(U(e\otimes\xi)|U(f\otimes\eta))=(e\otimes\xi |f\otimes\eta)\defequal
(\xi|\pi(<e,f>)\eta).$$
Here is the computation
$$\begin{array}{cl}
&(e\otimes\xi |f\otimes\eta)\\
&=(\xi|\pi(<e,f>)\eta)\\
&=\int (\xi\circ r(\gamma)|<e,f>_r(\gamma)\pi(\gamma)\eta\circ s(\gamma))\delta^{-1/2}(\gamma)d\lambda^x(\gamma)d\mu(x)\\
&=\int(e(\gamma'^{-1})|f(\gamma'^{-1}\gamma))(\xi\circ r(\gamma)|\pi(\gamma)\eta\circ s(\gamma))\delta^{-1/2}(\gamma)d\lambda^x(\gamma')d\lambda^x(\gamma)d\mu(x)\\
&=\int(e(\gamma'^{-1})|f(\gamma))(\xi\circ r(\gamma')|\pi(\gamma'\gamma)\eta\circ s(\gamma))\delta^{-1/2}(\gamma'\gamma)d\lambda^{s(\gamma')}(\gamma)d\lambda^x(\gamma')d\mu(x)\\
&=\int(e(\gamma')|f(\gamma))(\pi(\gamma')\xi\circ s(\gamma')|\pi(\gamma)\eta\circ s(\gamma))\delta^{-1/2}(\gamma)d\lambda^{r(\gamma')}(\gamma)d\nu_0(\gamma')\\
&=\int(U(e\otimes\xi)(x)|U(f\otimes\eta)(x))d\mu(x)\\
&=(U(e\otimes\xi)|U(f\otimes\eta))\\
\end{array}$$
Let us check next that $U$ intertwines the induced representation $\pi'$ and the integrated representation $\underline\pi$. This is straightforward: for $f\in C_c(G), e\in C_c(G,r^*E)$ and $\xi\in L^2(G^{(0)},H)$, we have
$$U(fe\otimes\xi)(x)=\int f(\gamma')L(\gamma')e({\gamma'}^{-1}\gamma)\otimes\pi(\gamma)\xi\circ s(\gamma)\delta^{-1/2}(\gamma)d\lambda^x(\gamma')d\lambda^x(\gamma)$$
On the other hand
$$\begin{array}{ccl}
\underline\pi(f)U(e\otimes\xi)(x)&=&\int f(\gamma')\underline\pi(\gamma')U(e\otimes\xi)(s(\gamma')\delta^{-1/2}(\gamma')d\lambda^x(\gamma')\\
&=&\int f(\gamma')L(\gamma')e(\gamma)\otimes\pi(\gamma\gamma')\xi\circ s(\gamma)\delta^{-1/2}(\gamma\gamma')d\lambda^{s(\gamma')}(\gamma)d\lambda^x(\gamma')\\
&=&\int f(\gamma')L(\gamma')e({\gamma'}^{-1}\gamma)\otimes\pi(\gamma)\xi\circ s(\gamma)\delta^{-1/2}(\gamma)d\lambda^x(\gamma')d\lambda^x(\gamma)\\
\end{array}$$
Let ${\mathcal H}'$ be the Hilbert space of the induced representation $\pi'$. It is the separation-completion of $C_c(G)\otimes L^2(G^{(0)},H)$ with respect to the scalar product
$$(e\otimes\xi |f\otimes\eta)=(\xi|\pi(<e,f>)\eta)$$
Let $j: C_c(G)\otimes L^2(G^{(0)},H)\ra H'$ the corresponding map.
The representation $\pi$ extends to the right C*-module $C^*(G,r^*E)$: to $e\in C_c(G,r^*E)$ we associate $\pi(e)\in{\mathcal B}(L^2(G^{(0)},H),{\mathcal H}')$ such that $\pi(e)\xi=j(e\otimes\xi)$. Then, we have for $f,g\in C_c(G,r^*E)$, $\pi(g)^*\pi(f)=\pi(<g,f>_r)$. Since the integrated representation $\underline\pi$ is bounded with respect to the full norm, so is the induced representation $\pi'$: for $f\in C_c(G)$, $\|\pi'(f)\|\le \|f\|_{\underline A}$. For $f\in C_c(G)$ and $e\in C_c(G,r^*E)$, we have $\pi(fe)=\pi'(f)\pi(e)$. This gives the inequality
$$\pi(<fe,fe>)\le \|f\|^2_{\underline A}\pi(<e,e>)$$
and $\|fe\|\le \|f\|_{\underline A}\|e\|$.
\sk
Note that if $H$ is a regular $G$-Hilbert bundle, i.e. of the form $H_x=L^2(Z_x,\alpha_x)$ where $(Z,\alpha)$ is a principal $G$-space with $r$-system $\alpha$ (some multiplicity is allowed), then $E\otimes H$ is also a regular $G$-Hilbert bundle. This shows that our result is also valid for the reduced C*-algebra $A=C_r^*(G)$.
\sk
A straightforward computation shows that multiplication on the left by $f^*$ is the adjoint of the multiplication on the left by $f$: for $f\in C_c(G)$ and $\xi,\eta\in C_c(G,r^*E)$, we have
$<f\xi,\eta>=<\xi,f^*\eta>$. We have shown $(f,\xi)\mapsto f\xi$ defines a $*$-homomorphism of $\underline A$ into ${\mathcal L}_A({\mathcal E})$. This $*$-homomorphism is non\-degenerate: using a left approximate identity for $C_c(G)$ for the inductive limit topology as in \cite[Proposition 2.1.9]{ren:approach}, one obtains that the linear space of the $f\xi$'s, where $f\in C_c(G)$ and $\xi\in C_c(G,r^*E)$, is dense in the inductive limit topology, hence in the norm topology.
This concludes the proof.
\end{proof}

In other words, $C^*(G,r^*E)$ is a $(C^*(G), C^*(G))$-correspondence. Similarly,
$C_r^*(G,r^*E)$ is a $(C_r^*(G), C_r^*(G))$-correspondence. Let us give some properties of this construction. We consider on one hand the category ${\mathcal H}(G)$ of $G$-Hilbert bundles and, on the other the category ${\mathcal C}(A)$ of $(A,A)$-correspondences, where $A$ is a C*-algebra.  The objects of ${\mathcal H}(G)$ are $G$-Hilbert bundles $E,F,\ldots$; its morphisms are bounded equivariant continuous linear bundle maps $\varphi: E\ra F$, where bounded means that $\sup_{x\in X}\|\varphi_x\|<\infty$. The objects of ${\mathcal C}(A)$ are  correspondences ${\mathcal E}, {\mathcal F},\ldots$ as defined above; its morphisms are $(A,A)$-bimodule continuous  linear maps $J: {\mathcal E}\ra {\mathcal F}$. Given Hilbert $G$-Hilbert bundles $E,F$, one can define the tensor product $E\otimes F$. It is a $G$-Hilbert bundle, where, for $\gamma\in G$, $e\in E_{s(\gamma)}$ and $f\in F_{s(\gamma)}$, $L_{E\otimes F}(\gamma)(e\otimes f)=L_E(\gamma)e\otimes L_F(\gamma)f$.

\begin{theorem}\label{functor} The map $E\mapsto C^*(G,r^*E)$ (resp. $C_r^*(G,r^*E)$)  is a covariant functor from the category of $G$-Hilbert bundles to the category of $(C^*(G), C^*(G))$-correspondences (resp. $(C_r^*(G), C_r^*(G))$-correspondences)  which, up to isomorphism, transforms the tensor product of $G$-Hilbert bundles into the composition of correspondences.
\end{theorem}

\begin{proof} Let $\varphi: E\ra F$ be a morphism from a $G$-Hilbert bundle $E$ to a $G$-Hilbert bundle $F$.  The corresponding morphism  is given for $\xi\in C_c(G,r^*E)$, by $\varphi_*\xi=\varphi\circ\xi$. If we identify $C^*(G,r^*E)$ with $C_0(X,E)\otimes C^*(G)$ and $C^*(G,r^*F)$ with $C_0(X,F)\otimes C^*(G)$ as in lemma \ref{Le Gall}, the map $\varphi_*$ is the map $\Phi\otimes 1$, where $\Phi$ is the $C^*$-module morphism from $C_0(X,E)\ra C_0(X,F)$ defined by $\Phi(\xi)=\varphi\circ\xi$. It is well known (see \cite[page 42]{lan:C*}) that it extends to a C*-module morphism $\varphi_*: C^*(G,r^*E)\ra C^*(G,r^*F)$. Since $\varphi_*(f\xi)=f\varphi_*(\xi)$ for all $f\in C_c(G)$ and $\xi\in C_c(G,r^*E)$, $\varphi_*$ is a morphism of correspondences. The functorial properties are clear. Let $E,F$ be $G$-Hilbert bundles. Let us show that $C^*(G,r^*(E\otimes F))$ and the composition $C^*(G,r^*E)\otimes_{C^*(G)}C^*(G,r^*F)$ are isomorphic. We define the map
$$j_c: C_c(G,r^*E)\otimes C_c(G,r^*F)\ra C_c(G,r^*(E\otimes F))$$
by $$j_c(\xi\otimes\eta)(\gamma)=\int\xi(\gamma')\otimes L_F(\gamma')\eta(\gamma'^{-1}\gamma)d\lambda^{r(\gamma)}(\gamma').$$
 One checks easily that for all
$\xi\in C_c(G,r^*E), \eta\in C_c(G,r^*F)$ and $f\in C_c(G)$,
$$j_c((\xi\otimes\eta)f)=j_c(\xi\otimes\eta)f\quad\hbox{and}\quad j_c(f(\xi\otimes\eta))=f j_c(\xi\otimes\eta).$$
A straightforward (but lengthy) computation using changes of variables and order of integration and the left invariance of the Haar system gives  that for all
$\xi\in C_c(G,r^*E)$ and $\eta\in C_c(G,r^*F)$,
$$<j_c(\xi\otimes\eta), j_c(\xi\otimes\eta)>=<\xi\otimes\eta,\xi\otimes\eta>$$
where the second inner product is defined after Definition \ref{C*-correspondence}. Therefore $j_c$ extends to an inner product preserving linear map
$$j: C^*(G,r^*F)\otimes_{C^*(G)} C^*(G,r^*F)\ra C^*(G,r^*(E\otimes F)).$$
Let us show that it is onto. It suffices to show that its image is dense. For that, it suffices to show that the image of $j_c$ is dense in $C_c(G,r^*(E\otimes F))$ in the inductive limit topology. This is true because, by definition of $j_c$, its image contains the sections of the form $\gamma\mapsto \xi\circ r (\gamma)\otimes f\eta (\gamma)$, where $\xi\in C_c(X,E), \eta\in C_c(G,r^*F)$ and $f\in C_c(G)$ and, as we have seen earlier, the linear span of the $f\eta$'s is dense in $C_c(G,r^*F)$. By continuity, $j$ is a $(C^*(G),C^*(G))$-bimodule map. The proof is identical when we use instead the reduced C*-algebra $C_r^*(G)$ and the reduced C*-module $C_r^*(G,r^*E)$.
\end{proof}

\section{Derivations}

\subsection{The derivation defined by a cocycle}

Let us add to the previous data, namely a locally compact groupoid with Haar system $(G,\lambda)$ and a $G$-Hilbert bundle $E$, a continuous cocycle $c: G\ra r^*E$. We shall see that $c$ defines a derivation with values in the bimodule $C^*(G,r^*E)$ (or $C_r^*(G,r^*E)$).

\begin{definition} Let $A$ be a Banach algebra and let $\mathcal E$ be an $(A,A)$ bimdodule. A derivation with values in $\mathcal E$ is a linear map
$\partial:{\mathcal A}\ra {\mathcal E}$, where $\mathcal A$ is a dense subalgebra of $A$ satisfying: 
$$\forall a,b\in{\mathcal A},\quad \partial(ab)=a\partial(b)+\partial(a)b.$$
\end{definition}

\begin{proposition} Let $G,E$ and $c$ be as above. Then the map $\partial: C_c(G)\ra C^*(G,r^*E)$ [resp. $C_r^*(G,r^*E)$] defined by $\partial(f)(\gamma)=if(\gamma)c(\gamma)$ for all $f\in C_c(G)$ is a derivation.
\end{proposition}

\begin{proof} This is a straightforward verification.
\end{proof}

\begin{proposition} The derivation $\partial: C_c(G)\ra C_r^*(G,r^*E)$ defined by the cocycle $c:G\ra r^*E$ is closable.
\end{proposition}

\begin{proof}
Let $\mu$ be a quasi-invariant measure for $(G,\lambda)$. We introduce the measure $\nu=\mu\circ\lambda$. Let $L_{\mu}$ be the left regular representation on $H=L^2(G,\nu^{-1})$ given by $L_\mu(f)g=f*g$. Consider also the Hilbert space $K=L^2(G,r^*E,\nu^{-1})$ of square integrable sections. We denote by $j:C_c(G)\ra H=L^2(G,\nu^{-1})$ and also by $j: C_c(G,r^*E)\ra K= L^2(G,r^*E,\nu^{-1})$ the linear embeddings. For $\xi\in C_c(G,r^*E)$, the map $j(g)\mapsto j(\xi g)$ extends to a bounded operator $L_\mu(\xi): L^2(G,\nu^{-1})\ra L^2(G,r^*E,\nu^{-1})$ and $L_\mu(\xi)$ is still defined for $\xi\in C_r^*(G,r^*E)$.  Multiplication by $ic$ defines an unbounded linear map 
$$\partial_\mu: C_c(G)\subset L^2(G,\nu^{-1})\ra L^2(G,r^*E,\nu^{-1})$$
such that $j\circ\partial=\partial_\mu\circ j$.
Since it is a multiplication operator, it is closable. 

Let $(f_n)$ be a sequence in $C_c(G)$ which tends to $0$ in $C_r^*(G)$ and such that $(\partial f_n)$ tends to $\xi$ in $C_r^*(G,r^*E)$. We will show that $\xi=0$.

Let $g\in C_c(G)$. We have:
$$\partial (f_n*g)=(\partial f_n)g+ f_n(\partial g).$$
Taking the image by the map $j: C_c(G,r^*E)\ra K= L^2(G,r^*E,\nu^{-1})$, we obtain
$$\partial_\mu\circ j(f_n*g)=j((\partial f_n)g)+ j(f_n(\partial g)).$$
The term $j((\partial f_n)g)=L_\mu(\partial f_n)j(g)$ tends to $L_\mu(\xi)j(g)$ and the term $j(f_n(\partial g))=L_\mu(f_n)j(g)$ tends to 0. Since $j(f_n*g)=L_\mu(f_n)j(g)$ tends to 0 and $\partial_\mu$ is closable, $L_\mu(\xi)j(g)=0$. This implies that $L_\mu(\xi)=0$. Since the family of representations $(L_\mu)$, where $\mu$ runs over all quasi-invariant measures, is faithful on $C_r^*(G,r^*E)$, we obtain $\xi=0$.

\end{proof}

\subsection{A construction of Sauvageot}

In the theory of non-commutative Dirichlet forms, J.-L. Sauvageot has given in the following construction of a derivation from a semi-group of completely prositive contractions. 

Let $(T_t)_{t\ge 0}$ be a strongly continuous semi-group of completely positive contractions  of a C*-algebra $A$. We denote by $-\Delta$ its generator (it is a dissipation). We assume that $\mathcal A$ is a dense sub $*$-algebra in its domain. The associated Dirichlet form is the sesquilinear map ${\mathcal L}:{\mathcal A}\times {\mathcal A}\ra A$ defined by
$${\mathcal L}(\alpha,\beta)={1\over 2}[\alpha^*\Delta(\beta)+\Delta(\alpha^*)\beta-\Delta(\alpha^*\beta)].$$
One can prove that ${\mathcal L}$ is completely positive in the sense that
$$\forall n\in\N, \forall \alpha_1,\ldots,\alpha_n\in{\mathcal A}, [{\mathcal L}(\alpha_i,\alpha_j)]\in M_n(A)^+.$$

The positivity of the non-commutative Dirichlet form gives the following GNS representation.
\begin{theorem}\cite{sau:tangent}\label{Sauvageot} Let $(T_t)_{t\ge 0}$ be a semi-group of CP contractions  of a C*-algebra $A$ as above and let $\mathcal L$ be its associated Dirichlet form. Then 
\begin {enumerate} 
\item There exists an $(A,A)$-correspondence $\mathcal E$ and a derivation $\partial:{\mathcal A}\ra {\mathcal E}$ such that
\begin{itemize}
\item for all $\alpha,\beta\in{\mathcal A}$, ${\mathcal L}(\alpha,\beta)=<\partial(\alpha),\partial(\beta)>_A$;
\item the range of $\partial$ generates $\mathcal E$ as a right Banach $A$-module.
\end{itemize}
\item  If $({\mathcal E}',\partial')$ is another pair satisfying the same properties, there exists a C*-bimodule isomorphism $u: {\mathcal E}\ra {\mathcal E}'$ such that $\partial'=u\circ \partial$.
\end{enumerate}
\end{theorem}

\subsection{CNT functions and Sauvageot's construction}

Let $(G,\lambda)$ be a locally compact groupoid and let $\psi: G\ra\R$ be a continuous function conditionally of negative type. For $t\ge 0$, define $T_t: C_c(G)\ra C_c(G)$ by $T_t(f)(\gamma)=\exp(-t\psi(\gamma))f(\gamma)$.

\begin{proposition} Given a continuous CNT function $\psi:G\ra \R$
\begin{enumerate}
\item Pointwise multiplication by $e^{-t\psi}$, where $t\ge 0$ defines a completely positive contraction $T_t:C^*(G)\ra C^*(G)$ and $(T_t)_{t\ge 0}$ is a strongly continuous semi-group of completely positive contractions of $C^*(G)$.
\item Its generator is $-\Delta$ with the dense sub-$*$algebra ${\mathcal A}=C_c(G)$ as essential domain and with $\Delta f=\psi f$.
\end{enumerate}
\end{proposition}

\begin{proof} Let us fix $t\ge 0$. Then $\varphi=e^{-t\psi}$ is a continuous function of positive type bounded by 1. It is well known (see \cite[Theorem 4.1]{rw:Fourier} or \cite[Proposition 2.3]{ren:Fourier}) that pointwise multiplication by $\varphi$ defines a completely positive linear map $T_t$ from $C^*(G)$ into itself with completely bounded norm not greater than 1. The semi-group property is clear. The strong continuity can be checked by looking at the coefficients of an arbitrary representation of $C^*(G)$. For $f\in C_c(G)$, $\displaystyle {e^{-t\psi}-1\over t}f$ tends to $-\psi f$ as $t\to 0^+$ in the inductive limit topology, hence in $C^*(G)$. To prove that ${\mathcal A}=C_c(G)$ is an essential domain, we use \cite[Theorem 1.5.2]{bra:der}: we check that for $\lambda>0$, $\{(1+\lambda\psi)f, f\in C_c(G)\}$ is dense in $A$, which is clear.
\end{proof}

Here is the expression of the Dirichlet form of this example: for $f,g\in C_c(G)$, we have
$${\mathcal L}(f,g)(\gamma)=\int {1\over 2}[\psi(\gamma)-\psi(\gamma'^{-1}\gamma)-\psi(\gamma')]\overline{f(\gamma'^{-1})}g(\gamma'^{-1}\gamma)d\lambda^{r(\gamma)}(\gamma').$$

We want to compare Sauvageot's construction with the GNS representation of Theorem \ref{GNS}. The latter construction yields a real Hilbert bundle; we complexify it and still denote it by $E$.

\begin{theorem}\label{Sauvageot pair} Up to isomorphism, the Sauvageot pair $({\mathcal E},\partial)$ constructed from the CNT function $\psi$ as above is $(C^*(G,r^*E), \partial')$, where $(E,c)$ is the GNS representation of $\psi$ given in Theorem \ref{GNS} and $\partial'$ is the associated derivation.
\end{theorem}

\begin{proof}
We compute the Dirichlet form associated to $(C^*(G,r^*E), \partial')$:
$$\begin{array}{ccl}
{\mathcal L}'(f,g)(\gamma)&=&<\partial'f,\partial' g>\\
&=&\int [(c(\gamma^{-1}|c(\gamma'^{-1}\gamma))]\overline{f(\gamma'^{-1})}g(\gamma'^{-1}\gamma)d\lambda^{r(\gamma)}(\gamma').\\
\end{array}$$
An easy computation gives the equality
$${1\over 2}[\psi(\gamma)-\psi(\gamma'^{-1}\gamma)-\psi(\gamma')]=(c(\gamma^{-1})|c(\gamma'^{-1}\gamma))$$
hence the equality of the Dirichlet forms ${\mathcal L}$ and ${\mathcal L}'$. Let us show that the range of $\partial'$ generates $C^*(G,r^*E)$ as a right C*-module. Consider the subspace $\Lambda=\{(fc)g: f, g\in C_c(G)\}$ of $C_c(G,r^*E)$. By definition, 
$$(fc)g(\gamma)=\int c(\gamma\gamma')f(\gamma\gamma')g(\gamma'^{-1})d\lambda^{s(\gamma)}(\gamma').$$
Since $c(G^{r(\gamma)})$ is total in $E_{r(\gamma)}$, 
$\Lambda(\gamma)=\{(fc)g(\gamma): f,g\in C_c(G)\}$ is dense in $E_{r(\gamma)}$. Moreover, for $f\in C_c(G)$ and $\xi\in\Lambda$, $\gamma\mapsto f(\gamma)\xi(\gamma)$ belongs to $\Lambda$.\cite[Proposition 14.1]{fd:representations} (or rather its proof) shows that $\Lambda$ is dense in $C_c(G,r^*E)$ in the inductive limit topology, hence in
$C^*(G,r^*E)$ in the C*-module topology. We conclude by Theorem \ref{Sauvageot} $(ii)$. 
\end{proof}

\bibliographystyle{amsplain}

\end{document}